\documentclass{article}

\usepackage{graphicx}
\usepackage{amsmath,graphicx,amsfonts,hyperref}

\usepackage{cite}
\usepackage[english]{babel}
\usepackage[ansinew]{inputenc}
\usepackage{float,latexsym,enumerate,amssymb,color}
\usepackage{subfig}

\newtheorem{theorem}{Theorem}
\newtheorem{lemma}[theorem]{Lemma}

\newenvironment{proof}{\begin{trivlist} \item[] {\it Proof.}}{\hfill $\Box$\end{trivlist}}

\newcommand{\disc}{\operatorname{disc}}

\begin{document}


\title{New results on the coarseness of bicolored point sets}

\author{
J. M. D\'{\i}az-B\'{a}\~{n}ez\thanks{Departamento de Matem\'atica Aplicada II, 
Universidad de Sevilla, Seville, Spain. \{dbanez,iventura\}@us.es.
Partially supported by project 
FEDER MEC MTM2009-08652, and the ESF EUROCORES programme 
EuroGIGA -ComPoSe IP04-MICINN Project EUI-EURC-2011-4306.}
\and
R. Fabila-Monroy\thanks{CINVESTAV, Instituto Polit\'ecnico Nacional, Mexico. Email: ruyfabila@math.cinvestav.edu.mx. Partially supported by grant 153984 (CONACyT, Mexico).}
\and
P. P\'erez-Lantero\thanks{Escuela de Ingenier\'ia Civil en Inform\'atica, Universidad de Valpara\'{i}so, Chile. Email: pablo.perez@uv.cl. Partially supported by grant CONICYT, FONDECYT/Iniciaci\'on 11110069 (Chile).}
\and
I. Ventura\footnotemark[1]
}

\maketitle

\begin{abstract}
Let $S$ be a 2-colored (red and blue) set of $n$ points in the plane.
A subset $I$ of $S$ is an island if
there exits a convex set $C$ such that $I=C\cap S$. 
The \emph{discrepancy} of an island is the absolute value
of the number of red minus the number of blue points
it contains.
A \emph{convex partition} of $S$ is a partition of $S$ into
islands with pairwise disjoint convex hulls. The discrepancy of a convex partition
is the discrepancy of its island of minimum discrepancy. The \emph{coarseness}
of $S$ is the discrepancy of the convex partition of $S$
with maximum discrepancy. This concept was 
recently defined by Bereg et al.\ [CGTA 2013]. 
In this paper we study the following problem:
Given a set $S$ of $n$ points in general position in the plane,
how to color each of them (red or blue) such that
the resulting 2-colored point set has small coarseness?
We prove that every $n$-point set $S$ can be colored such that
its coarseness is $O(n^{1/4}\sqrt{\log n})$. This bound is almost tight
since there exist $n$-point sets such that every 2-coloring 
gives coarseness at least $\Omega(n^{1/4})$.
Additionally, we show that there exists an approximation algorithm
for computing the coarseness of a 2-colored point set, whose ratio is between $1/128$
and $1/64$, solving an open problem posted by Bereg et al.\ [CGTA 2013].
All our results consider $k$-separable islands of $S$, for some $k$,
which are those resulting from intersecting $S$ with at most $k$
halfplanes. 
\end{abstract}

\section{Introduction}\label{intro}

Let $S$ be a finite set of $n$ elements, 
and ${\cal Y}\subseteq 2^{S}$ be a family of subsets of $S$.
The tuple $(S,{\cal Y})$ is called a \emph{range space}. 
If the range space arises from point sets and geometric objects, 
then $(S,{\cal Y})$ is called a \emph{geometric range space}. 
A \emph{coloring} of $S$ is a mapping ${\cal X}: S \rightarrow \{-1, +1\}$.
We think of the elements of $S$  mapped to $-1$ as being colored \emph{blue} and 
the elements of $S$ mapped to $+1$ as being colored \emph{red}. 
Let $R$ be the red elements of $S$ and $B$ its blue elements.
For $Y \subseteq S$, let ${\cal X}(Y) := \sum_{y\in Y}{\cal X}(y)$.
The \emph{discrepancy} of $Y$ is defined as $\disc(Y)=|{\cal X}(Y)|$,
that is, the absolute value
of the number of red minus the number of blue points
that $Y$ contains.
The \emph{discrepancy} of the family ${\cal Y}$ is defined as 
$\disc({\cal Y}):=\min_{\cal X} \max_{Y\in {\cal Y}} \disc(Y)$. 

Geometric discrepancy theory
has applications in statistics, clustering, optimization, 
and computer graphics. 
See the textbooks~\cite{alexander,pach,chazelle,drmota-tichy,matousek-disc}
for problems and results in geometric discrepancy.
 
Assume from now on that $S$ is a set of $n$ points in general 
position in the plane.
A subset $I$ of $S$ is called an \emph{island} if 
there is a convex set $C$ on the plane such that $I=C\cap S$~\cite{islas}.
A \emph{convex partition} of $S$ is a partition of $S$ into islands, with
pairwise disjoint convex hulls.

Given a coloring of $S$,
the {\em discrepancy} of a convex partition $\Pi=\{S_1,$ $S_2,\ldots,S_k\}$ of $S$, 
denoted by $\disc(\Pi)$, 
is the minimum of $\disc(S_i)$ for $i=1,\ldots,k$. 
The \emph{coarseness} of $S$, denoted by $\mathcal{C}(S)$, is defined as the maximum
of $\disc(\Pi)$ over all the convex partitions $\Pi$ of $S$.
This concept of coarseness was just recently defined by Bereg et al.~\cite{coarseness},
as a parameter to measure how well blended a finite set $R$ of red points and a 
finite set $B$ of blue points are.
The smaller $\mathcal{C}(R\cup B)$ the more blended $R$ and $B$~\cite{coarseness}.

Suppose now that we have an $n$-point set $S$ in the plane and want to color
each of its elements (red or blue) such that the resulting 2-colored point set
has high coarseness. The answer to this question is trivial: take a halving
line, color the points to one side red, and color the other points blue; or even easier,
color all points of the same color, say red.

Then we post the following question: 
{\em What is the smallest coarseness of $S$
over all colorings of $S$?}

In this paper we show that for every $n$-point set $S$ in the plane 
there exists a coloring of $S$ such that the coarseness of $S$ is upper bounded
by $O\left ( n^{1/4} \sqrt{\log n} \right)$ (Theorem~\ref{thm:upper}).
We also show that there exist point sets
such that all colorings give
coarseness at least $\Omega(n^{1/4})$ (Theorem~\ref{thm:lower}).
We prove the upper bound by showing that the discrepancy
of a convex partition is
closely related to the discrepancy of a certain class
of islands of $S$, which we call $k$-\emph{separable} islands.

Given a finite point set $S$ in the plane and a coloring of $S$, 
computing the coarseness of $S$
is believed to be NP-hard~\cite{coarseness}. We also show for the first time that
there exists a polynomial-time constant approximation algorithm. Its approximation
ratio is between $1/128$ and $1/64$, depending on 
$\disc(S)=\big||R|-|B|\big|$ (Theorem~\ref{theo:approx}). Specifically,
the approximate value of the coarseness that we provide is at least
$\max\left\{\frac{\mathcal{C}(S)}{128},\frac{\mathcal{C}(S)}{64}-\disc(S)\right\}$
and at most $\mathcal{C}(S)$.
With this result, we solve an
open problem posted by Bereg et al.~\cite{coarseness}.

\section{Visiting the discrepancy theory}\label{sec:disc-theory}

In this section we recall some definitions and results 
from discrepancy theory.

The \emph{primal shatter function} $\pi_{\cal Y}(m)$ of $(S,{\cal Y})$
is a function of $m$. It is defined as the 
maximum number of subsets into which 
a subset of $S$, of at most $m$ elements, can be split (or ``shattered'')
by all the elements of ${\cal Y}$. Formally:

$$\pi_{\cal Y}(m):=\max_{A\subset S,|A|\le m}|\{Y\cap A : Y \in {\cal Y}\}| $$

The \emph{dual shatter function} $\pi_{\cal Y}^*(m)$ is obtained
by exchanging the roles of the points in $S$ with the sets
in ${\cal Y}$. $\pi_{\cal Y}^*(m)$ is defined as the maximum
number of equivalence classes on $S$ defined by an 
$m$-element subfamily ${\cal Z} \subset {\cal Y}$, where
 two elements $x$ and $y$ of $S$ are equivalent if they
belong  to the same sets of ${\cal Z}$.

The primal and dual shatter functions have been used
to give tight and almost tight upper bounds on the discrepancy of
range spaces, via the following theorems (see Chapter~5 of \cite{matousek-disc}).

\begin{theorem}{\bf (Primal shatter function bound).}\label{thm:primal}
Let $d>1$ and $C$ be constants such that $\pi_{\cal Y}(m) \le Cm^d $
for all $m \le n$. Then $\disc(\mathcal{Y})$ is upper bounded by $O\left ( n^{1/2-1/2d} \right )$.
\end{theorem}

\begin{theorem}{\bf (Dual shatter function bound).} \label{thm:dual}
Let $d>1$ and $C$ be constants such that $\pi_{\cal Y}^*(m) \le Cm^d $
for all $m \le |{\cal Y}|$. Then $\disc(\mathcal{Y})$ is upper bounded by 
$O\left ( n^{1/2-1/2d}\sqrt{\log n} \right )$.
\end{theorem}

For example, if ${\cal H}$ is the family of halfplanes,
the it is easy to see that $\pi_{\cal H}(m)=O(m^2)$. Thus the discrepancy
of halfplanes is $O(n^{1/4})$. It is known that this
bound is tight:

\begin{lemma}\label{lem:halfspaces}{\bf(\cite{alex_halfspaces,chaz_halfspaces})}
For arbitrarily large values of $n$, there exist sets of 
$n$ points in general position in the plane such that, given any coloring of $S$,
a halfplane exists within which one color outnumbers the other
by at least $Cn^{1/4}$, for some positive constant $C$.
\end{lemma}

From Lemma~\ref{lem:halfspaces} we prove
the following Theorem:

\begin{theorem} \label{thm:lower}
For arbitrarily large values of $n$, there exist sets of 
$n$ points in general position  in the plane with
coarseness at least $Cn^{1/4}$ for some positive constant $C$.
\end{theorem}

\begin{proof}
Assume that $S$ is a set of points as in Lemma~\ref{lem:halfspaces},
and consider any coloring of $S$.
Thus, there exists a halfplane $H$ such that $\disc(S\cap H)\ge C' n^{1/4}$
for some positive constant $C'$.
Suppose that the trivial convex partition $\{S\}$ has discrepancy
at most $(C'/2)n^{1/4}$, as otherwise we are done with $C:=C'/2$.
Then we have that $\disc(S\setminus H)\ge (C'/2) n^{1/4}$ and
the convex partition of $\Pi:=\{S \cap H, S \setminus H\}$ of $S$ 
has discrepancy $\disc(\Pi)\ge (C'/2)n^{1/4}$. Thus, the coarseness
of $S$ is at least $Cn^{1/4}$, with $C:=C'/2$.
\end{proof}

\section{$k$-separable islands and convex partitions}\label{sec:k-sep}

An island $I$ of $S$ is $k$-{\em separable} if it can be 
separated from $S\setminus I$ with at most $k$ halfplanes,
that is, there exist halfplanes $H_1,H_2,\ldots,H_t$ $(1\leq t\leq k)$, 
such that $I=S \cap (H_1\cap H_2 \cap \ldots H_t)$. 
We denote the family of all the $k$-separable islands of $S$ with ${\cal I}_k$.
For constant $k$, we upper bound the discrepancy of ${\cal I}_k$
by using its dual shatter function. Namely, we show that 
$\pi_{{\cal I}_k}^*(m)=O(m^2)$. We should point out that Dobkin and Gunopulos~\cite{dobkin1995}
proved the same asymptotic upper bound, but 
our proof considers more details and explicitly gives the constant 
hidden in the big-O notation.


\begin{lemma}\label{lem:conv}
If $k$ is a positive integer and $S$ a set of $n$ points in convex and general
position in the plane then $\pi_{{\cal I}_k}^*(m)\le 4km$.
\end{lemma}

\begin{proof}
Assume that $S$ is sorted clockwise around its convex hull. 
Note that any $k$-separable island must consist of at most $k$ intervals
of consecutive points of $S$ in this order. Consider a family of $m$,
$k$-separable islands. There are at most $2km$ points of $S$ that are the endpoints
of any such intervals. 
There are at most $2km$ regions into which the remaining points
(which are not endpoints of any interval) can lie.
Thus, in total there are at most $4km$ equivalence classes.
\end{proof}

\begin{lemma}\label{lem:general}
If $k$ is a positive integer and $S$ a set of $n$ points in general position 
in the plane then $\pi_{{\cal I}_k}^*(m)\le (k^2+4k)m^2$.
\end{lemma}

\begin{proof}
Let $\mathcal{F}$ be a family of $m$, $k$-separable islands on $S$.
We first consider the points lying in the convex hull 
of some island $I$  of $\mathcal{F}$.
Note that the convex hull of $I$ is a set of points in convex position. By
Lemma~\ref{lem:conv} these points are in at most $4k(m-1)$ different 
equivalence classes (when considering the other $m-1$ islands in $\mathcal{F}$).
Thus, in total there at most $4km^2$ equivalence classes for points in the boundary
of some island in $\mathcal{F}$. We now bound the number of equivalence
classes for points not lying in the boundary of any island. 
Each such equivalence class is contained in a cell of 
the line arrangement defined by the following
set of lines $\mathcal{L}$.
For each island $I \in \mathcal{F}$, let $L_I$ be the set
of at most $k$ lines that separate $I$ from $S\setminus{I}$.
Set $\mathcal{L}:=\cup_{I \in \mathcal{F}} L_I$. 
The line arrangement defined by $\mathcal{L}$ has at most 
$|\mathcal{L}|^2=k^2m^2$ cells. 
The result thus follows.
\end{proof}

Using the dual shatter function bound and Lemma~\ref{lem:general}
we obtain the following theorem:

\begin{theorem}\label{thm:sep} 
Let $k$ be a positive constant and $S$ a set
of $n$ points in general position in the plane. The discrepancy
of the family $\mathcal{I}_k$ of the $k$-separable islands of $S$ 
is upper bounded by $O\left ( n^{1/4}\sqrt{\log n} \right )$.
\end{theorem}

Note that although $k$-separable islands have small discrepancy, this
is not the case for islands in general. For example for any coloring of a set of $n$
points in convex position in the plane there always exists an island
with discrepancy at least $n/2$. It can
be shown that in this case the primal and dual shatter function are
equal to $2^m$. 

We now show that every convex
partition must contain a $5$-separable island. This follows immediately from:

\begin{lemma}\label{lem:edels}{\bf(Theorem~2 in~\cite{edels})}
A collection of $n$ compact, convex, and pairwise disjoint sets in 
the plane may be covered with $n$ non-overlapping convex polygons
with a total of not more than $6n-9$ sides.
\end{lemma}

\begin{theorem}\label{thm:5sep}
Every convex partition $\Pi$ of $S$ has a $5$-separable island.
\end{theorem}

\begin{proof}
Let $\Pi:=\{S_1,S_2\dots,S_m\}$. Using Lemma~\ref{lem:edels},
there exist non-overlapping convex polygons $C_1,C_2,\dots, C_m$ with
a total of no more than $6m-9$ sides, such that for each $i=1,\ldots,m$ the 
convex hull of $S_i$ is enclosed
by $C_i$. Thus, one of these convex polygons has at most $5$ sides and the
enclosed island is thus a $5$-separable island.
\end{proof}

We arrive at our main result by combining 
Theorems~\ref{thm:5sep} and~\ref{thm:sep}.

\begin{theorem}\label{thm:upper}
For every set $S$ of $n$ points in general position
in the plane there exists a coloring such that
the coarseness of $S$ is upper bounded by $O(n^{1/4}\sqrt{\log n})$.
\end{theorem}

\section{Approximation}\label{sec:approximation}

Let $S=R\cup B$ be finite point set in the plane in general position, and
let ${\cal X}$ be a coloring of $S$,
where $R$ is the set of points colored red
and $B$ the set of points colored blue. 
Let $r:=|R|$, $b:=|B|$, and $D_k:=\max_{I\in\mathcal{I}_k}\disc(I)$. 
For every set $X\subseteq\mathbb{R}^2$, let $\overline{X}$ denote the complement
of $X$, that is, $\overline{X}=\mathbb{R}^2\setminus X$.

In this section we show that the value of $D_2$ is a constant approximation
for the coarseness of $S$. We start with some lemmas before arriving to the
final result.

\begin{lemma}\label{lem:1-island}
Let $t$ be an integer. If there exists an island $I\in\mathcal{I}_1$ of $S$
such that $\disc(I)\ge t$,
then there exists a convex partition $\Pi$ of $S$ such that 
$$\disc(\Pi)\ge\max\left\{t/2,t-|r-b|\right\}.$$
\end{lemma}

\begin{proof} 
We have that $\disc(S\setminus I)\ge t-|r-b|$. Indeed, 
\begin{eqnarray*}
\disc(S\setminus I) & = & \Bigl|\bigl(r-|I\cap R|\bigr)-\bigl(b-|I\cap B|\bigr)\Bigr|\\
& = & \Bigl|\bigl(|I\cap R|-|I\cap B|\bigr)-\bigl(r-b\bigr)\Bigr|\\
& \ge & \bigl||I\cap R|-|I\cap B|\bigr|-\bigl|r-b\bigr|\\
& = & \disc(I)-\bigl|r-b\bigr|\\
& \ge & t-|r-b|.
\end{eqnarray*}
If $t-|r-b|\ge t/2$ then for the convex partition $\Pi=\{I,S\setminus I\}$ we have that $\disc(\Pi)\ge t-|r-b|$.
Otherwise, if $t-|r-b|<t/2$, then $\disc(S)=|r-b|>t/2$ which implies that $\disc(\Pi)>t/2$
for the trivial convex partition $\Pi=\{S\}$. The result thus follows.
\end{proof}

\begin{lemma}\label{lem:2-island}
Let $t$ be an integer. If there exists an 
island $I\in\mathcal{I}_2$ of $S$ such that $\disc(I)\ge t$,
then there exists a convex partition $\Pi$ of $S$ such that 
$$\disc(\Pi)\ge\max\left\{t/8,t/4-|r-b|\right\}.$$
\end{lemma}

\begin{proof}
If $I\in\mathcal{I}_1$ then the result follows from Lemma~\ref{lem:1-island}.
Thus consider that $I\in\mathcal{I}_2\setminus\mathcal{I}_1$, and
let $H_1$ and $H_2$ be two halfplanes such that $I=S\cap(H_1\cap H_2)$.
Let $I':=S\cap(\overline{H_1}\cap H_2)$, $I'':=S\cap(H_1\cap\overline{H_2})$, 
and $I''':=S\cap(\overline{H_1}\cap\overline{H_2})$. 
Refer to Figure~\ref{img:fig1}.

\begin{figure}[h]
	\centering
	\includegraphics[scale=0.65]{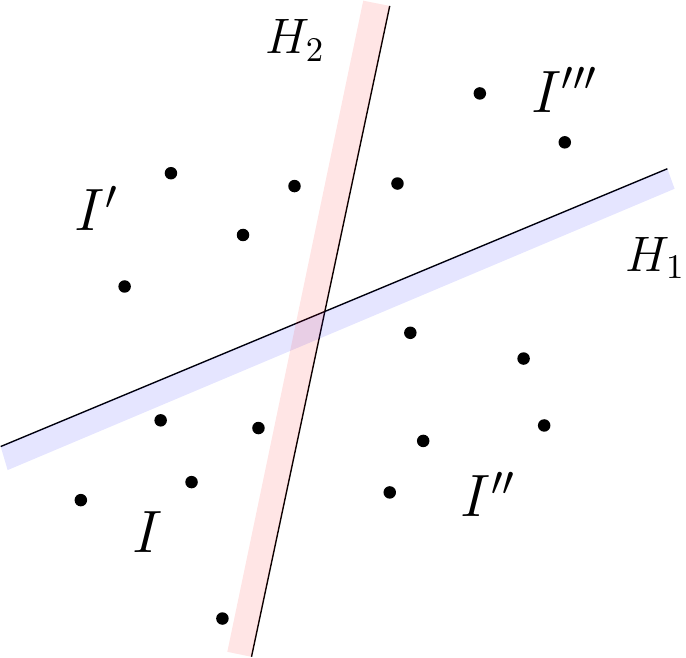}
	\caption{\small{Proof of Lemma~\ref{lem:2-island}: The island $I$ belongs to 
	$\mathcal{I}_2\setminus\mathcal{I}_1$, then it is the intersection of $S$
	with two halfplanes $H_1$ and $H_2$.}}
	\label{img:fig1}
\end{figure}

If $\disc(I')\leq t/2$
then the island $I\cup I'\in\mathcal{I}_1$ satisfies $\disc(I\cup I')\ge t/2$ and
by Lemma~\ref{lem:1-island} there exists a convex partition $\Pi_1$ such that
\begin{equation}\label{eq1}
\disc(\Pi_1)\ge\max\left\{t/4,t/2-|r-b|\right\}.
\end{equation}
The same happens if $\disc(I'')\leq t/2$.
Otherwise, if $\disc(I')>t/2$ and $\disc(I'')>t/2$, then we proceed as follows.
If $\disc(I''')\ge t/4$ then the convex partition $\Pi_2=\{I,I',I'',I'''\}$ satisfies
\begin{equation}\label{eq2}
\disc(\Pi_2)\ge t/4.
\end{equation}
Otherwise, we have that the island $I'\cup I'''\in\mathcal{I}_1$ satisfies $\disc(I'\cup I''')>t/4$ and then,
by Lemma~\ref{lem:1-island}, there exists a convex partition $\Pi_3$ such that
\begin{equation}\label{eq3}
\disc(\Pi_3)\ge\max\left\{t/8,t/4-|r-b|\right\}.
\end{equation}
Combining equations (\ref{eq1}-\ref{eq3}) the result follows.
\end{proof}

\begin{lemma}\label{lem:Dkmas1-Dk}
$D_3\leq 4 D_2$, and $D_{k+1}\leq 2D_k$ for $k\geq 3$.
\end{lemma}

\begin{proof}
Let an island $I\in\mathcal{I}_3$ such that $D_3=\disc(I)$. If $I\in\mathcal{I}_2$ then
$D_3=D_2$ since $\mathcal{I}_2\subseteq\mathcal{I}_3$. Otherwise, 
we have $I\in\mathcal{I}_3\setminus\mathcal{I}_2$. Then, let $H_1$, $H_2$, and $H_3$
be three halfplanes such that $I=S\cap (H_1\cap H_2\cap H_3)$. 
Let $I':=S\cap(H_1\cap H_2\cap\overline{H_3})$ and $I'':=S\cap(\overline{H_1}\cap\overline{H_3})$,
and observe that $I\cup I'$, $I'\cup I''$, and $I''$ belong to $\mathcal{I}_2$.
Refer to Figure~\ref{img:fig2}.

\begin{figure}[h]
	\centering
	\includegraphics[scale=0.65]{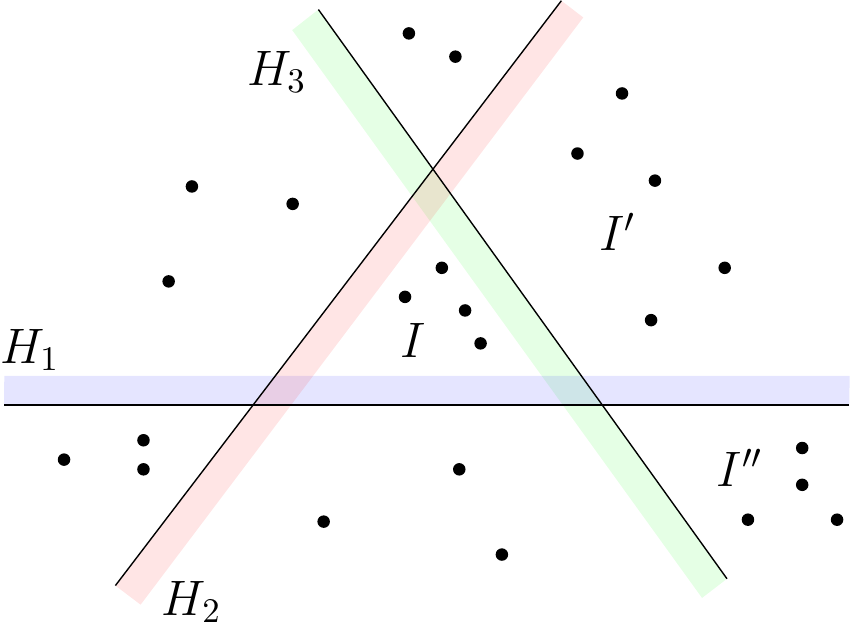}
	\caption{\small{Proof of Lemma~\ref{lem:Dkmas1-Dk}: The island $I$ belongs to 
	$\mathcal{I}_3\setminus\mathcal{I}_2$, then it is the intersection of $S$
	with three halfplanes $H_1$, $H_2$, and $H_3$.}}
	\label{img:fig2}
\end{figure}

If $\disc(I')\leq D_3/2$ then we have
$D_3/2\le \disc(I\cup I')\le D_2$ which implies $D_3\le 2D_2$.
Otherwise, if $\disc(I')> D_3/2$, we proceed as follows. 
If $\disc(I'')\ge D_3/4$ then we have $D_3\le 4D_2$. 
Otherwise, if $\disc(I'')< D_3/4$, then $\disc(I'\cup I'')>D_3/4$ implying $D_3<4D_2$. 
Then we have proved $D_3\leq 4 D_2$.

To prove the other part of the lemma, let $I\in\mathcal{I}_{k+1}$ such that 
$\disc(I)=D_{k+1}$. If $I\in\mathcal{I}_k$ then
$D_{k+1}=D_k$ since $\mathcal{I}_k\subseteq\mathcal{I}_{k+1}$. Otherwise, if 
$I\in\mathcal{I}_{k+1}\setminus\mathcal{I}_k$, let $H_1,H_2,\ldots,H_{k+1}$ be
$k+1$ halfplanes such that $I=S\cap(H_1\cap H_2\cap\ldots\cap H_{k+1})$. 
Refer to Figure~\ref{img:fig3}.

\begin{figure}[h]
	\centering
	\includegraphics[scale=0.65]{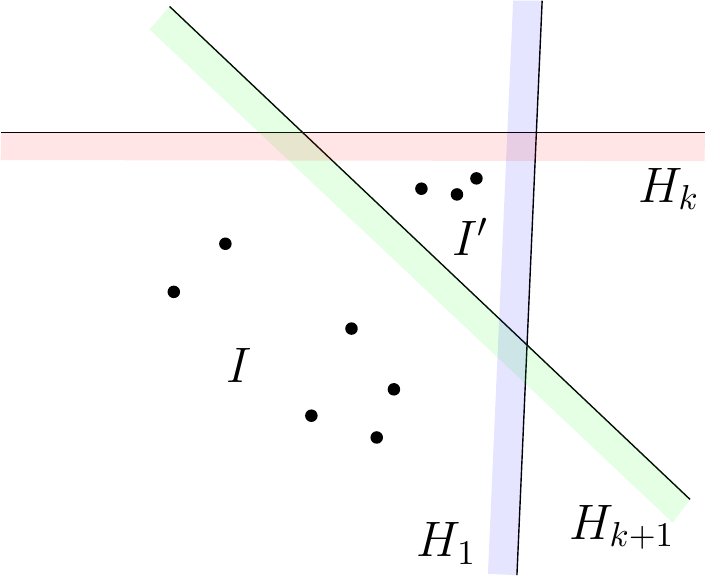}
	\caption{\small{Proof of Lemma~\ref{lem:Dkmas1-Dk}: The island $I$ belongs to 
	$\mathcal{I}_{k+1}\setminus\mathcal{I}_k$, then $I$ is the intersection of $S$
	with $k+1$ halfplanes $H_1,\ldots,H_{k+1}$. The island $I'$ is such that $I\cup I'$
	is the intersection of $S$ with the halfplanes $H_1,\ldots,H_{k}$.}}
	\label{img:fig3}
\end{figure}

Assume w.l.o.g.\ that
the edges of the polygon $H_1\cap H_2\cap\ldots\cap H_{k+1}$ in clockwise order
belong to the boundary of $H_1,H_2,\ldots,H_{k+1}$,
respectively. 
Let $I':=S\cap(H_k\cap\overline{H_{k+1}}\cap H_1)\in\mathcal{I}_3$, and
observe that $I\cup I'\in\mathcal{I}_k$.
If $\disc(I')\ge D_{k+1}/2$ then $D_{k+1}\le 2\disc(I')\le 2D_3\le 2D_{k}$.
Otherwise, we have that $\disc(I\cup I')>D_{k+1}/2$ which implies that $D_{k+1}<2D_k$.
\end{proof}

\begin{theorem}\label{theo:approx}
$\max\left\{\frac{\mathcal{C}(S)}{128},\frac{\mathcal{C}(S)}{64}-|r-b|\right\}
\leq\max\left\{\frac{D_2}{8},\frac{D_2}{4}-|r-b|\right\}\leq\mathcal{C}(S)$.
\end{theorem}

\begin{proof}
Observe that $\max\left\{\frac{D_2}{8},\frac{D_2}{4}-|r-b|\right\}\leq\mathcal{C}(S)$
follows from Lemma~\ref{lem:2-island}. Since any convex partition of $S$
has a 5-separable island, and using Lemma~\ref{lem:Dkmas1-Dk}, we have
that $\mathcal{C}(S)\leq D_5\leq 2D_4\leq 4D_3\leq 16D_2$. With these facts 
the result follows.
\end{proof}

\begin{theorem}
There exists a polynomial time constant-approximation algorithm
for computing $\mathcal{C}(S)$.
\end{theorem}

\begin{proof}
The value of $D_2$, equal to the 
discrepancy of a $2$-separable island of maximum discrepancy,
can be computed in $O(n^3\log n)$ time~\cite{dobkin1995}. The result
then follows from Theorem~\ref{theo:approx}. 
\end{proof}

\section{Conclusions}

We proved that the discrepancy of the family of all $k$-separable islands 
is upper bounded
by $O(n^{1/4}\sqrt{\log n})$, by showing that its dual shatter function
$\pi_{\mathcal{I}_k}^*(m)$ is upper bounded by $O(m^2)$. 
It is known that 
the dual shatter function bound can be tight for some range spaces
(see~\cite{matousek-disc}). It is not hard to see that the 
primal shatter function of the $k$-separable islands of point sets 
in convex position
is lower bounded by $\Omega(m^k)$.
So the primal shatter function bound can be arbitrarily worse than
the dual shatter function bound in this case. It is also interesting
to note that the discrepancy of $1$-separable islands (or halfplanes)
is upper bounded by $O(n^{1/4})$. We leave the exact (asymptotic)
computation of the discrepancy of $k$-separable islands
as an open problem.

Using the fact that every convex partition of a point set $S$ has 
an island (in this case a $5$-separable island)
of low discrepancy, we showed that every $n$-point set
in general position in the plane can be two-colored so that the coarseness is
upper bounded by $O(n^{1/4}\sqrt{\log n})$. However, Theorem~\ref{lem:edels}
provides more information; for any positive constant $c<1$ there exists
a positive integer $k_c$ (depending only on $c$), so that in
every convex partition of $S$ into $m$ islands at least $cm$ of them
are $c_k$-separable (and thus have small discrepancy). 
We think that computing the exact asymptotic value
of the above bound on the coarseness of point sets is an interesting 
(and hard) open problem.

We further showed the first approximation
for computing the coarseness of a colored point set,
which is believed to be NP-hard. Our approximation is based on a known
algorithm. Proving the hardness of this
problem remains open, and giving improved approximations as well.

\section{Acknowledgments}

The problems studied here were introduced and partially solved 
during a visit to University of Valparaiso funded by project Fondecyt 11110069 (Chile).
The authors would like to thank an anonymous referee for helpful comments.

\small

\bibliographystyle{plain}
\bibliography{disc-islands}

\end{document}